\RequirePackage{fix-cm}
\documentclass[smallextended]{svjour3}       

\smartqed  
\usepackage{graphicx}
\usepackage{mathptmx}      
\usepackage{subcaption}
\usepackage{amsmath}
\usepackage{amssymb}
\usepackage{hyperref}
\makeatletter
\let\cl@chapter\undefined
\makeatletter
\usepackage{cleveref}
\usepackage{xcolor}
\definecolor{pec}{HTML}{FEC04C}

 \usepackage{amsthm}
 
\usepackage{thmtools, thm-restate}
\usepackage{tabularx}
\usepackage{todonotes}
\usepackage{optidef}
\usepackage{algpseudocode}
\usepackage{algorithm}
\usepackage{tikz}
\usetikzlibrary{hobby, shapes.geometric, arrows, arrows.meta, calc}
\usepackage{tikz-3dplot}
\usepackage{authblk}

\usepackage[style=numeric,maxnames=10]{biblatex}
\usepackage{bm}
\usepackage{booktabs}
\newcommand{\real}{\mathbb{R}}
\newcommand{\complex}{\mathbb{C}}
\newcommand{\jim}{\bm{i}}
\newcommand\M{\mathcal{M}}%

\newcommand{\EFF}{E_{\mathrm{FF}}}
\newcommand{\Einc}{E^{\mathrm{inc}}}
\newcommand{\Eout}{E^\mathrm{out}}

\newtheorem{fact}{Fact}
\crefname{fact}{fact}{facts}
\Crefname{fact}{Fact}{Facts}

\renewcommand\Re{\operatorname{Re}}%
\renewcommand\Im{\operatorname{Im}}%
\renewcommand\div{\operatorname{div}}%
\newcommand\curl{\operatorname{curl}}%
\newcommand\supp{\operatorname{supp}}%
\newcommand\conv{\operatorname{conv}}%

\newcommand{\C}{\mathbb{C}}
\newcommand{\R}{\mathbb{R}}

\newcommand{\V}{\mathcal{V}}

\newcommand*\dd{\mathop{}\!\mathrm{d}}

\begin{document}
\tdplotsetmaincoords{60}{110}
\pgfmathsetmacro{\rvec}{.8}
\pgfmathsetmacro{\thetavec}{30}
\pgfmathsetmacro{\phivec}{60}

\title{Solving A PDE-Constrained Antenna Design Problem By Mixed-Integer Programming
\thanks{This research is conducted as part of the Dutch National Growth Fund programme \mbox{POLARIS} (www.polaris-ngf.nl).
}
}


\author{Ruben Hoeksma         \and
        Jesse van Rhijn       \and
        Marc Uetz             \and
        Wick Wijnholds
}


\institute{R.P. Hoeksma \at
              University of Twente, Faculty of Electrical Engineering, Mathematics and Computer Science \\
              \email{r.p.hoeksma@utwente.nl}           
           \and
           J. van Rhijn \at
              University of Twente, Faculty of Electrical Engineering, Mathematics and Computer Science, and\\
              TNO, Radar Technology Department, The Hague\\
              \email{j.vanrhijn@utwente.nl}
           \and
           M. Uetz \at
              University of Twente, Faculty of Electrical Engineering, Mathematics and Computer Science\\
              \email{m.uetz@utwente.nl}
           \and
           W.J. Wijnholds \at
              University of Twente, Faculty of Electrical Engineering, Mathematics and Computer Science\\
              \email{w.j.wijnholds@utwente.nl}
}

\date{}

\maketitle

\begin{abstract}
This paper shows that a mixed integer PDE-constrained design optimization problem can be made amenable to be solved to optimality using standard mixed-integer programming solvers. The concrete problem that we consider is the optimal design of the topology of an antenna, a  problem that keeps haunting
the engineering community. 
Indeed, up to this point no exact method was known. Existing methods  rely on iterative use of PDE-solvers, embedded in either gradient methods or metaheuristics. However, none of these methods come with performance guarantees. We break this barrier, showing that important antenna design problems can be solved optimally  using a mixed integer linear program, plus one singleton second order cone constraint  in two variables. The main contribution of the paper is to develop such a formulation. There are two major methodological steps to get there.
The first is to show how a standard boundary element method (BEM), based on a discretized design space, translates into a mixed integer program at all. 
Approached intuitively, this yields a maximization problem with non-concave objective, which is hence not tractable.
The second is to make the objective function linear, using a new Charnes--Cooper transform for mathematical programming over complex numbers.
We evaluate our approach on the classic dipole antenna, reconfirming well-known experimental results, and apply it to a widely-used test case, the Yagi-Uda antenna. In a successor paper, we report on further computational enhancements using physics-informed cuts.
\end{abstract}

\section{Introduction}
\label{section: introduction}
Electromagnetic waves are all around us; from the radio waves used to transmit information to visible light, UV light, X-rays and beyond. When these waves encounter matter, they interact with it, to be deflected, blocked or reflected. The various ways in which this occurs are just as diverse as the ways in which we observe and utilize electromagnetic waves, and are collectively known as \emph{scattering}. For a schematic representation, see \Cref{fig: blob}. The focus of this work is to describe a method by which the shape of an object that scatters electromagnetic waves, a \emph{scatterer}, can be optimally chosen from a given design space.
To this end, we set up a mixed-integer program (MIP), that schematically takes the form
\begin{align*}
    \max \quad& (\text{electric field amplitude in some direction}) \\
    \text{s.t.} \quad & (\text{electric field equations are satisfied}), \\
    \quad & (\text{design constraints are satisfied}),
\end{align*}
where we have chosen as our objective the maximization of scattered electric field in some direction. This is consistent with the application that 
this work is actually based on:
the optimal design of the topology of an antenna.

It quickly turns out that this is a challenging undertaking: First, the object of interest is the strength of the emitted electric field in the distance, and the design variables describe the topology of an antenna. However the link between
these two quantities is not explicitly given. It is implicitly described by Maxwells equations, which gives rise to an inverse problem. We resolve this issue and show how a standard boundary element method (BEM) can be used to 
derive a first mixed integer program; this is discussed in \Cref{section: basic model}. Second, 
the na\"ive approach still leads to an intractable  nonconvex mixed-integer program. 
To overcome this, in \Cref{section: charnescooper} we show how to transform the na\"ive MIP into a convex mixed-integer program with linear objective  using a novel approach that is reminiscent of the Charnes--Cooper transformation for linear fractional programming.~\cite{charnes1962programming}. Indeed, we show that our approach can be combined with that of Charnes and Cooper to obtain a more general result, applicable to linear fractional programming over the complex numbers.

There are two research areas which relate to what we do in this paper. 
First and foremost, this is an antenna topology optimization problem. Finding good antenna designs has been the subject of decades of research. Various methods have been developed; they can roughly be divided into two main classes~\cite{lucchini2022topology}:
\begin{enumerate}
    \item Gradient-based methods: this research focuses on finding derivatives of, e.g., the size of certain features of an antenna. This gradient is then fed into an optimization algorithm to find a local optimum. This approach has the significant downside that the topology of the antenna is fixed; no holes can be added or removed. Also, finding these derivatives can be a laborious process; the found derivatives are often specific to the type of antenna.
    \item Metaheuristic methods: a general-purpose metaheuristic like a genetic algorithm is employed to search for a `good' antenna. The main advantage of these methods is that they are very flexible; the main downside is that, by their nature, these algorithms provide no guarantee of solution quality, nor an evident stopping criterion. Since they usually rely on a penalty term in the objective function to impose constraints, solutions provided by these methods may not satisfy all of the constraints.
\end{enumerate}

In contrast, MIPs allow us to find globally optimal solutions subject to hard constraint, with a known stopping criterion; even if terminated prematurely, they can provide bounds on the quality of possible solutions.

Second, the problem that we solve here is an example of a class of problems also known as Mixed-Integer PDE-Constrained Optimization (MIPDECO). 
In these problems, one searches for the optimal (partially) integer-valued (and often infinite-dimensional) input to a system described by one or more PDEs.~\cite{leyffer2015MIPDECO}
These optimization procedures are often not exact, as exact methods are rarely feasible for real-world problems~\cite{leyfferConvergenceSumupRounding2021a,sharmaInversionConvectionDiffusion2021,WESSELHOEFT2018799}.

In contrast, we here demonstrate that an exact method is feasible for a practically relevant class of PDE-constrained design problems, namely the design of optimal 2D antenna topologies. Specifically, in this paper we evaluate our MIP for two test cases: the classic dipole antenna, and the Yagi-Uda antenna. In the former case, we reconfirm existing literature values on its optimal length. In the latter case, we show how one can apply our method to a nontrivial design problem and obtain an optimal solution in reasonable computation time.

This paper is the first of a series of works on this topic. In our second work~\cite{hoeksma2026physics} we discuss a series of extensions to the method we present here, which enable the design of more complex antennas subject to additional, practical constraints. Together with these two fundamental works, a companion paper~\cite{vanrhijn2026optimal} showcases the benefits of the method with more demanding antenna design applications; that companion paper has the purpose to advertise the approach for the antenna engineering community using three practical use cases.

The structure of this paper is as follows:

First, in \Cref{section: prelims}, we build up the electromagnetics background needed to construct (instances of) our MIP,
Then, in \Cref{section: formulation}, we derive the aforementioned Charnes--Cooper-like transform, and apply it to our problem to obtain a feasible MIP,
In \Cref{section: results}, we then test our MIP on two antenna design problems,
Finally, \Cref{section: discussion} contains some concluding remarks.
\section{Preliminaries}
\label{section: prelims}
\subsection{Notation}
\label{section: notation}
To avoid confusion with various summation indices, the imaginary unit is denoted by a boldface $\jim$, so $\jim^2 = -1$.
For a complex number $z \in \complex$, we write its complex conjugate as $\bar z$,
and we denote the Hermitian adjoint (that is, the complex conjugate transpose) of a matrix $A \in \complex^{n \times m}$ as $A^H$.\footnote{
    While these are the standard notations in the mathematics literature, physics and engineering works often denote the complex conjugate by $(\cdot)^*$ and the Hermitian adjoint by $(\cdot)^\dagger$.
}
We use the standard component-wise partial order for vectors; that is, if $x, y \in \real^n$, then $x \leq y$ means $x_i \leq y_i$ for all $i \in [n]$.
For complex vectors we extend this in the natural way: if $x, y \in \complex^n$, then $x \leq y$ means $\Re\{x_i\} \leq \Re\{y_i\}$ and $\Im\{x_i\} \leq \Im\{y_i\}$ for all $i \in [n]$.

Fix a Cartesian coordinate system $(x, y, z)$ for $\real^3$.
It is often more convenient to work in spherical coordinates $(r, \vartheta, \varphi)$, where
\begin{align*}
    x &= r\sin\vartheta\cos\varphi, \\
    y &= r\sin\vartheta\sin\varphi, \\
    z &= r \cos\vartheta.
\end{align*}
Here, $r = \sqrt{x^2 + y^2 + z^2}$ is the Euclidean distance from the origin, $\vartheta \in [0, \pi]$ denotes the polar angle with respect to the $z$-axis, and $\varphi \in [0, 2\pi)$ is the azimuthal angle with respect to the $x$-axis.

\subsection{Electromagnetism}
\label{section: electromagnetism}

The interactions that govern electromagnetic phenomena (including scattering) are captured in Maxwell's equations (see, e.g., \cite{griffithsIntroductionElectrodynamics2023}), which describe the relationships between currents, charges, and electric and magnetic fields, given here in their time-harmonic form:
\begin{align}
    \div E &= \frac{\rho}{\epsilon_0}\label{eq: Maxwell1}\\
    \div B &= 0\label{eq: Maxwell2}\\
    \curl E &= - \jim \omega B\label{eq: Maxwell3}\\
    \curl B &= \mu_0 \left(J +\jim \omega \epsilon_0 E\right),\label{eq: Maxwell4}
\end{align}
where 
\begin{itemize}
\item $J, \rho: \R^3 \rightarrow \C^3$ are the current and charge density, 
\item $E, B: \R^3 \rightarrow \C^3 $ are the electric and magnetic field, 
\item $\omega \in \R_{>0}$ is the angular frequency at which we evaluate these equations,
\item $\mu_0, \epsilon_0 \in \R_{>0}$ are known constants, and
\item The operators $\div$ and $\curl$ are the standard linear differential operators from vector calculus.
\end{itemize}

Maxwell's equations describe several different closely related phenomena. Intuitively, one can see them as describing the following behavior:

Charges generate electric fields and electric fields move charges~\eqref{eq: Maxwell1}, magnetic fields travel in loops~\eqref{eq: Maxwell2}, currents can generate magnetic fields and vice versa~\eqref{eq: Maxwell4}, and electric and magnetic fields can interact with each other to propagate as electromagnetic waves (e.g., light, radio waves)~\eqref{eq: Maxwell3},~\eqref{eq: Maxwell4}.

In this paper, we will be working mostly with currents and electric fields - in particular, we are interested in the interactions that these have that generate propagating electromagnetic waves. These interactions are collectively known as scattering (see \Cref{fig: blob}). They are what happens when an electric field encounters an object: the electric field will induce currents in this object, which in turn will generate their own electric and magnetic fields, which, under the right conditions, will propagate as waves.

We will use antennas as a running example of an electromagnetic scatterer; the analysis we provide, however, will be applicable to more general scatterers.

To optimize an antenna, we are interested in maximizing the amplitude of the electric field radiated by the antenna in the so-called far-field region in some prescribed direction. 
The connection between input, i.e., the voltage that we apply, which is given through the electric field $\Einc$ that is impressed on the antenna, and the object of interest, which is the electric field $\Eout$ radiated by the antenna, is unfortunately only given implicitly.  
To obtain $\Eout$, we therefore have to solve an inverse problem. 

We therefore briefly need to introduce the electrodynamics quantities that are necessary to describe what links input electric field $\Einc$ and output $\Eout$. First, $\Einc$ is given through the externally injected voltage. Note that for this paper, this is considered fixed. Our decision variables will be exclusively the geometry of the antenna. The link between $\Einc$ and $\Eout$ is the current density $J$ on the antenna. Their relationship is described through two linear operators $\mathcal{K}$ and $\mathcal{L}$. To formulate these, we model the geometry of an antenna as a bounded domain $\Omega$ of perfect electric conductor (PEC) with surface $\partial \Omega$. Due to this assumption, nonzero current densities exist only on $\partial \Omega$. 

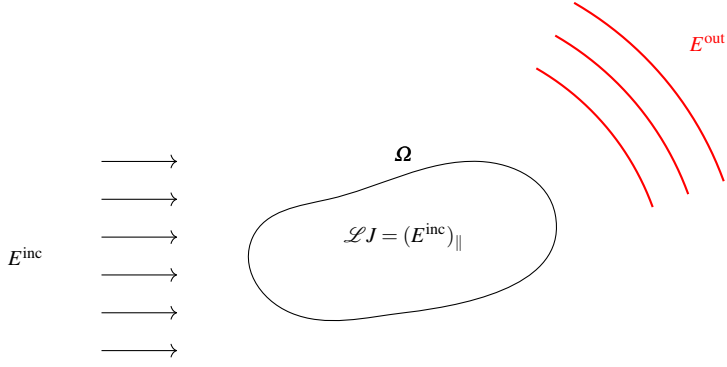
\begin{figure}
    \centering
    \begin{tikzpicture}
        \path[draw,use Hobby shortcut,closed=true]
            (0,0) .. (1,.5) .. (3,1) .. (4,.3) .. (2,-1) .. (.5,-1);
            \node[] (x) at (2, 1.1) {$\Omega$};

            \node at (2,1.1) {$\Omega$};
           
            \foreach \i in {-1, -0.5, 0, 0.5, 1, 1.5} {
                \draw[->] (-2, -\i) -- (-1, -\i);
            }

            \node at (2, 0) {$\mathcal{L}J = (\Einc)_\parallel$};
            
            \node[] (y) at (-3, -0.25) {$\Einc$};

            \draw [red,thick,domain=20:60] plot ({2+3.5*cos(\x)}, {-0.8+3.5*sin(\x)});
            \draw [red,thick,domain=20:60] plot ({2+4*cos(\x)}, {-0.8+4*sin(\x)});
            \draw [red,thick,domain=20:60] plot ({2+4.5*cos(\x)}, {-0.8+4.5*sin(\x)});
            
            \node[red] (Eout) at ({2+5.25*cos(40)}, {-0.8+5.25*sin(40)}) {$\Eout$};

    \end{tikzpicture}
    \caption{
        A general electromagnetic scattering problem.
        The incident field $\Einc$ induces a current density $J$ on the body $\Omega$, which radiates a field $\Eout$.
        In the case of an antenna, $\Einc$ is provided by a time-harmonic voltage source.
    }
    \label{fig: blob}
\end{figure}

\begin{remark}
    Although this section is written in terms of a three-dimensional conducting body $\Omega$ with two-dimensional boundary $\partial\Omega$, from \Cref{section: discrete geometry} onward we solely consider two-dimensional, infinitely thin antennas.
    Letting $\Sigma$ denote the two-dimensional antenna surface, the theory here still applies as long as one replaces both $\Omega$ and $\partial\Omega$ by $\Sigma$. This is an approach that is often taken in the antenna optimization community, as it generally provides a very good approximation of the behavior of actual antennas (with finite thickness), since the currents on these run only in a very thin surface layer.
\end{remark}

\begin{remark}
    While we model our scatterer as a PEC, this is not necessary to obtain a boundary element method (BEM) formulation. BEM formulations have been developed for many different material compositions, including dielectrics and layered media~\cite[Chapter 1.3]{davidsonComputationalElectromagneticsRF2010}.
\end{remark}

From Maxwell's equations, one can then derive the following; see, e.g., \cite{volakis2012IntegralEquationMethods}.
\begin{fact}\label{fact: EFIE operator}
	There exists a linear integrodifferential operator $\mathcal{L}: H^{1/2}_{\div}(\partial \Omega) \rightarrow H^{1/2}_{\div}(\partial \Omega)$
    that maps the current density $J$ on the surface $\partial \Omega$ of a bounded domain $\Omega \subset \R^3$ of PEC to the tangential component of the incident electric field on $\partial \Omega$, that is: 
	\begin{equation}
		\label{eq:EFIE}
	\mathcal{L} J = (\Einc)_{||}, 
	\end{equation}
	where $(\Einc)_{||}$ is the component of $\Einc$ tangential to $\partial \Omega$.
\end{fact}

\begin{fact}
\label{fact: far field operator}
	There exists a linear integral operator $\mathcal{K}: H^{1/2}_{\div}(\partial \Omega) \rightarrow H^1_{\mathrm{loc}}(\R^3\setminus \bar{\Omega})^3$ that maps the current density $J$ on the boundary $\partial \Omega$ of a domain $\Omega$ of PEC to the electric field $\Eout$ outside of the closure $\bar{\Omega}$ of $\Omega$ generated by these currents, that is: \begin{equation}
		\label{eq:Eout}
		 \Eout = \mathcal{K} J.
	\end{equation}
\end{fact}

The Sobolev spaces $H^{1/2}_{\div}(\partial \Omega)$ and $H^1_{\mathrm{loc}}(\R\setminus \bar{\Omega})$ are so chosen to guarantee that the solutions obtained have bounded energy---that is, that they abide by the laws of physics. 
For the definition of these spaces, see \cite{hsiao1997mathematical}. For a comprehensive treatment of such spaces, see, e.g., \cite{dautrayMathematicalAnalysisNumerical1990}.

We will lay out how these facts can be used to obtain a solution method to compute $J$ and $\Eout$ in \Cref{section: bem}.

As alluded to earlier, we are interested in the so-called \textit{antenna pattern}, which consists of the  \textit{far-field} electric field values $\EFF(\theta, \phi): \R^3 \rightarrow \C^3$ for a given set of angles (in spherical coordinates) $\vartheta \in [-\pi, \pi), \varphi \in [0, \pi]$, generally defined as follows:
Let $\hat{r}_{\vartheta, \varphi}$ be the spherical unit vector with angles $\vartheta, \varphi$, and $R$ be a radial distance from the origin. Then

\[
\EFF(\vartheta, \varphi) := \lim_{R\to\infty} R \Eout(R\hat{r}_{\vartheta, \varphi}).
\]
Note that the electric field $\Eout$ contains components that scale differently with radial distance: ones that scale with $1/R^3$, $1/R^2$ and $1/R$. In practice, we are only interested in the $1/R$ component. Hence $\EFF(\vartheta, \varphi)$ represents the electric field received at a sufficient distance in some direction, where all $o(1/R)$ field components vanish.
It is this antenna pattern that we use in the objective of our optimization problem. 

\subsection{The Boundary Element Method}
\label{section: bem}
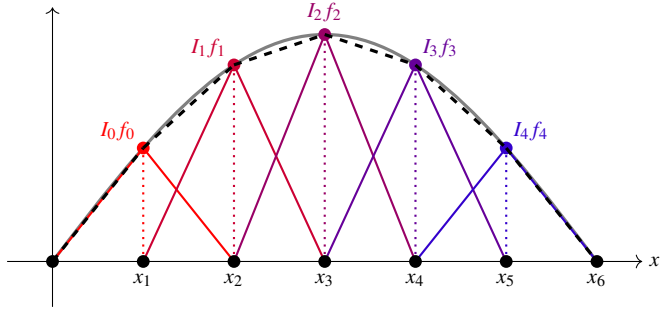
\begin{figure}
    \centering
    \begin{tikzpicture}[scale=1.2] 
\draw[->] (-0.5,0) -- (6.5,0) node[right] {$x$}; 
\draw[->] (0,-0.5) -- (0,2.8); 

\draw[gray, very thick] (0,0) sin (3, 2.5) cos (6, 0);
\def\coeff{{{1, {2.5*sin(pi/6 r)}}, {2, {2.5*sin(2*pi/6 r)}}, {3, {2.5*sin(3*pi/6 r)}}, {4, {2.5*sin(4*pi/6 r}}, {5, {2.5*sin(5*pi/6 r)}}}}
\foreach \i/\c in 
    {0/red!100!blue,
     1/red!80!blue,
     2/red!60!blue,
     3/red!40!blue,
     4/red!20!blue}
    {\draw[\c, thick] ({\coeff[\i][0] - 1}, 0) 
            -- (\coeff[\i][0], \coeff[\i][1]) 
            -- ({\coeff[\i][0] + 1}, 0);
    \fill[\c] (\coeff[\i][0], \coeff[\i][1]) circle (2pt);
    \ifnum \i < 2
        \node[above left, \c] at (\coeff[\i][0], \coeff[\i][1]) {$I_{\i}{f_\i}$};
    \fi
    \ifnum \i = 2
        \node[above=2pt, \c] at (\coeff[\i][0], \coeff[\i][1]) {$I_{\i}{f_\i}$};
    \fi
    \ifnum \i > 2
        \node[above right, \c] at (\coeff[\i][0], \coeff[\i][1]) {$I_{\i}{f_\i}$};
    \fi
    \draw[thick, dotted, \c] (\coeff[\i][0], 0) -- (\coeff[\i][0], \coeff[\i][1]);
    }

\draw[very thick, dashed] (0,0) -- (\coeff[0][0], \coeff[0][1]);
\draw[very thick, dashed] (\coeff[4][0], \coeff[4][1]) -- (6, 0);
\foreach \i in {0, ..., 3}
    \def\j{\i+1}
    \draw[very thick, dashed] (\coeff[\i][0], \coeff[\i][1]) -- (\coeff[\j][0], \coeff[\j][1]);

\foreach \x in {0, ..., 6} \fill (\x,0) circle (2pt); 

\foreach \x/\label in { 1/{x_{1}}, 2/{x_2}, 3/{x_{3}}, 4/{x_{4}}, 5/{x_5}, 6/{x_6} } 
{ \draw (\x,0.05) -- (\x,-0.05); \node[below] at (\x,-0.05) {$\label$}; } 
\end{tikzpicture}
    \caption{Example: Five overlapping scaled `tent' basis functions (each colored different) sum together to a piecewise-linear function (dashed), approximating $J$ (in gray), a half-period sine wave.}
    \label{fig: tent_function}
\end{figure}

When we are given the potential (or voltage) applied to the input port of our antenna (which is on $\partial\Omega$), we can use this to find $\Einc$. Our aim is to solve \cref{eq:EFIE} for $J$, and then substitute $J$ in \cref{eq:Eout} to find the antenna pattern. To do so, we use  the \textit{boundary element method} (BEM), which is one of the standard methods in the field of antenna engineering~\cite{davidsonComputationalElectromagneticsRF2010}. We will briefly recap how BEM works.

BEM is a computational technique whereby the boundary $\partial \Omega$ of a bounded domain $\Omega$ is discretized, and $N \in \mathbb{N}$ so-called \emph{basis functions} $f_i, i = 1, \ldots, N$ and \emph{testing functions} $g_j, j = 1, \ldots, N$ are defined on this discretization.  Here, think of piecewise linear functions living on subintervals of a line as in \Cref{fig: tent_function}, or in 3D, functions living on triangles of a triangulation of a surface, as in \Cref{fig: mesh}. 
Often, one chooses $f_i = g_i,~i=1, ..., N$, which is known as a \textit{Galerkin formulation}.

Each of these functions is supported by a small number of elements of the discretization. We then approximate our quantity of interest (in this case the current density $J$) as a weighted  sum of the basis functions, i.e.:
\begin{equation}
	J \approx \sum_{i=1}^N I_i f_i,
\end{equation}
where $I \in \C^N$.

We reformulate eq. \eqref{eq:EFIE} in a weaker form by using the testing functions. To do this, we define a suitable inner product $\langle \cdot, \cdot \rangle: H^{1/2}_{\div}(\partial \Omega) \times H^{1/2}_{\div}(\partial \Omega) \rightarrow \C$ as:

\begin{equation}
	\langle f, g\rangle = \int_{\partial \Omega} f(x') \cdot \overline{g(x')} dx'
\end{equation}
and  then use this to reformulate our problem by `testing' the left and right hand side of \eqref{eq:EFIE} using each of the testing functions $g_j$:
\begin{equation}
	\langle \mathcal{L} J, g_j\rangle = \langle (\Einc)_{||}, g_j\rangle, ~j = 1, ..., N
\end{equation}
We will now attempt to find expansion coefficients $I$ such that $\sum_{i=1}^N I_i f_i$ satisfies this set of equations, in other words:
\[
\langle\mathcal{L}\sum_{i=1}^N I_i f_i, g_j\rangle = \langle \mathcal{L} J, g_j\rangle = \langle (\Einc)_{||}, g_j \rangle, ~j = 1, ..., N. 
\]
Using linearity of $\mathcal{L}$ and substituting the above formulation, we obtain:
\begin{align*}
	\langle\mathcal{L}\sum_{i=1}^N I_i f_i, g_j\rangle &= \int_{\partial \Omega} \mathcal{L} \left(\sum_{i=1}^N I_i f_i(x')\right)\cdot \overline{g_j(x')} dx'\\
	&= \sum_{i=1}^N I_i  \int_{\partial \Omega} \mathcal{L} (f_i(x')) \cdot\overline{g_j(x')} dx'\\
    &= \sum_{i=1}^N I_i \langle\mathcal{L} f_i, g_j\rangle\\
	\langle n \times \Einc, g_j\rangle &= \int_{\partial \Omega} (n \times \Einc)(x') \cdot \overline{g_j(x')} dx'
\end{align*}
Therefore, defining
\begin{equation}\label{eq: Z definition}
	Z_{ij} = \langle \mathcal{L} f_i,g_j\rangle, ~i, j=1, ..., N
\end{equation}
and
\begin{equation}\label{eq: V definition}
	V_j =  \langle (\Einc)_{||}, g_j\rangle, j = 1, ..., N
\end{equation}
we obtain the linear system 
\begin{equation}
\label{eq: linear system MoM}
	ZI = V\,.
\end{equation}
Note that one can interpret \cref{eq: linear system MoM} as a weak version of \cref{eq:EFIE}, dependent on the chosen basis and testing functions. 
Here, note that current densities $J$ have been replaced by current density coefficients $I$, based on our choice of basis functions.

If we use a Galerkin formulation (so $g_i = f_i$), then we can use an additional property of the operator $\mathcal{L}$ from \Cref{fact: EFIE operator}, namely that
\[
\langle\mathcal{L} f_i, f_j \rangle = \langle \mathcal{L} f_j, f_i \rangle,
\]
which implies that $Z_{ij} = Z_{ji}$, so $Z$ is symmetric~\cite{rao1982electromagnetic}.

Recalling \Cref{fact: far field operator}, we apply $\mathcal{K}$ to our approximation of $J$ to obtain the electric field $\Eout$:
\begin{equation}\label{eq: far field coefficients}
	\Eout \approx \mathcal{K}\sum_{i=1}^N I_i f_i= \sum_{i=1}^N I_i \mathcal{K}f_i
\end{equation}
For a given $\vartheta, \varphi$, we then find the far-field values as
\begin{align*}
\EFF(\vartheta, \varphi) &= \lim_{R\to\infty} R E^\text{out}(R\hat{r}_{\vartheta, \varphi}) \approx \lim_{R\to\infty} R\sum_{i=1}^N I_i (\mathcal{K}f_i)(R\hat{r}_{\vartheta, \varphi})\\
&= \sum_{i=1}^N I_i \left(\lim_{R\to\infty} R(\mathcal{K}f_i)(R\hat{r}_{\vartheta, \varphi})\right)\,.
\end{align*}

Defining 
\[
    c_i(\vartheta, \varphi) = \lim_{R\to\infty} R(\mathcal{K}f_i)(R\hat{r}_{\vartheta, \varphi}), \quad i= 1, \ldots, N
\]
yields
\[
    \EFF(\vartheta, \varphi) \approx c(\vartheta, \varphi)^T I,
\]
so that $\EFF(\vartheta, \varphi)$ is a linear combination of $I$. Here, it is important to realize that the coefficients $c_i(\vartheta, \varphi)$ depend only on the geometry, its discretization and the choice of basis functions.

\subsection{Discretized Antenna Geometry}
\label{section: discrete geometry}

With the machinery described in \Cref{section: electromagnetism} and \Cref{section: bem} in place, all that remains in this section is to discuss the geometry discretization required to apply the BEM.
The usual approach in electromagnetic modeling is to approximate the boundary of the radiating body under study by a polygonal mesh with triangular faces~\cite{davidsonComputationalElectromagneticsRF2010}.
This allows for a great deal of flexibility~\cite{rao1982electromagnetic}, although one must take some care that such a triangulation is sufficiently refined in ``critical'' regions such as 
sharp corners.
\Cref{fig: mesh} shows an example of a triangulation of some 3D shape.

Since the triangulation is central to the underlying numerical modeling problem, we need a data structure to capture the geometry of the triangulation.
The standard approach is to use a \emph{mesh}, which is a 3-uniform hypergraph with some underlying geometry.

\begin{definition}
    A \emph{mesh} $\M = (\V, T)$ is a set $\V = \{v_1, v_2, \ldots, v_{|\V|}\}$ of vertices $v_i \in \real^3$ and a set $T = \{t_1, t_2, \ldots, t_{|T|}\}$ of triangles $t \in \V^3$, such that $t_i = \{v_{i_1}, v_{i_2}, v_{i_3}\}$ describes a triangle with vertices $v_{i_1}$, $v_{i_2}$ and~$v_{i_3}$. Also write $T(\M)$ to denote $\M$'s triangles.
\end{definition}

Given a mesh, one specifies the set of basis functions $\{f_i\}_{i=1}^N$ and test functions $\{g_i\}_{i=1}^N$ to discretize the current density as explained in \Cref{section: bem}.
The standard choice is the \emph{Rao--Wilton--Glisson} (RWG) basis~\cite{rao1982electromagnetic}.
Let $t_+$ and $t_-$ be two triangles sharing the edge $e = \{v_i, v_j\} = t_+ \cap t_-$.
Let $\ell_e = \|v_i - v_j\|$ be the length of $e$, and let $A_\pm$ be the area of $t_\pm$.
Lastly, let $v_\pm$ be the vertex of $t_\pm$ opposite the shared edge $e$.
Then the RWG function associated with edge $e$ is
\begin{align}
\label{eq: rwg}
    f_e(x) = \begin{cases}
        \frac{\ell_e}{2A_+} (x - v_+), & \text{ if } x \in \conv{t_+},\\
        \frac{\ell_e}{2A_-} (v_- - x), & \text{ if } x \in \conv{t_-},\\
        0, & \text{ otherwise},
    \end{cases}
\end{align}
This is a piecewise linear vector-valued function with the property that $f_i(x)$ is tangent to $t_\pm$ for $x \in t_\pm$.
See \Cref{fig: rwg} for a graphical representation of the RWG basis function.
We write for an RWG function $f$ the set of triangles on which $f$ is supported as $\supp(f)$.

One shortcoming of the standard RWG basis set is that the functions have zero normal component on the boundary of a given mesh.
As a consequence, it is not possible to use them to model a current density that enters a domain from an external source.
This can be overcome by augmenting a basis set with a ``half-RWG'' function~\cite{ulrich2002efficient}.
If an edge $e$ lies on the boundary of the mesh, one associates with $e$ a basis function $f_e^\text{half}$, which is simply the standard RWG basis function with only one active branch of \Cref{eq: rwg}.

With these basis functions in place, it remains to choose a set of test functions $\{g_i\}_{i=1}^N$.
The most common choice, and the one we make here, is a Galerkin formulation (so $g_i = f_i$, see also \Cref{section: bem}).

\begin{figure}
\centering
\begin{subfigure}[t]{0.4\linewidth}
\centering
        \begin{tikzpicture}
        [
          x={(1cm,0cm)}, y={(-0.6cm,-0.35cm)}, z={(0cm,1cm)},
          scale=0.4
        ]
        
        \tikzset{stlfacet/.style={draw=black, line join=round}}
        
        \draw[red!50, fill] (2.6,2.6,4) -- (2,2,5) -- (2.4839,1.3486,4.1935) -- cycle;
        \draw[red!50, fill] (2,2,5) -- (2.6,2.6,4) -- (1.1335,2.5127,4.1456) -- cycle;
        
        \input{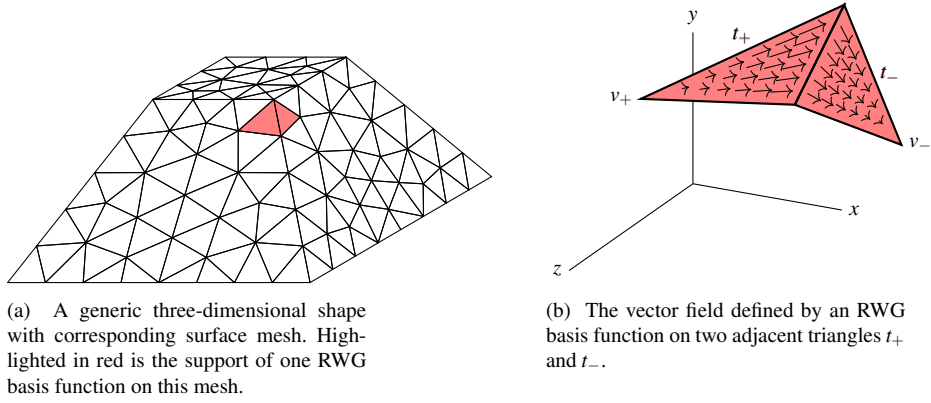}
        \end{tikzpicture}
\caption{
\label{fig: mesh}
    A generic three-dimensional shape with corresponding surface mesh.
    Highlighted in red is the support of one RWG basis function on this mesh.
}
\end{subfigure}
\hfill
\begin{subfigure}[t]{0.4\linewidth}
\centering
\begin{tikzpicture}[x={(-10:1cm)},y={(90:1cm)},z={(215:1cm)}]

    \def\length{0.08};
    \draw (2,0,0) node[right] {$x$} -- (0,0,0);
    \draw (0,0,0) -- (0,2,0) node[above] {$y$};
    \draw (0,0,0) -- (0,0,2) node[left] {$z$};

    \def\vX{-0.7}; \def\vY{1}; \def\vZ{0};
    \def\uX{2.8}; \def\uY{1}; \def\uZ{0};

    \def\eoneX{1.2}; \def\eoneY{2}; \def\eoneZ{-1};
    \def\etwoX{2.2}; \def\etwoY{2}; \def\etwoZ{1};

    \coordinate (v) at (\vX, \vY, \vZ);
    \node[left] at (v) {$v_+$};
    
    \coordinate (u) at (\uX, \uY, \uZ);
    \node[right] at (u) {$v_-$};
    
    \coordinate (e1) at (\eoneX, \eoneY, \eoneZ);
    \coordinate (e2) at (\etwoX, \etwoY, \etwoZ);
    
    \draw[red!50, fill] (barycentric cs:v=1,e1=0,e2=0) -- node[above] {$t_+$} (barycentric cs:v=0,e1=1,e2=0) -- (barycentric cs:v=0,e1=0,e2=1) -- cycle;
    \draw[red!50, fill] (barycentric cs:u=1,e1=0,e2=0) -- node[right] {$t_-$} (barycentric cs:u=0,e1=1,e2=0) -- (barycentric cs:u=0,e1=0,e2=1) -- cycle;
    \draw[thick] (barycentric cs:v=1,e1=0,e2=0) -- node[above] {$t_+$} (barycentric cs:v=0,e1=1,e2=0) -- (barycentric cs:v=0,e1=0,e2=1) -- cycle;
    \draw[thick] (barycentric cs:u=1,e1=0,e2=0) -- node[right] {$t_-$} (barycentric cs:u=0,e1=1,e2=0) -- (barycentric cs:u=0,e1=0,e2=1) -- cycle;
    
    \foreach \a/\b/\c in {
        0.125/0.125/0.750,
        0.125/0.250/0.625,
        0.125/0.375/0.500,
        0.125/0.500/0.375,
        0.125/0.625/0.250,
        0.125/0.750/0.125,
        0.250/0.125/0.625,
        0.250/0.250/0.500,
        0.250/0.375/0.375,
        0.250/0.500/0.250,
        0.250/0.625/0.125,
        0.375/0.125/0.500,
        0.375/0.250/0.375,
        0.375/0.375/0.250,
        0.375/0.500/0.125,
        0.500/0.125/0.375,
        0.500/0.250/0.250,
        0.500/0.375/0.125,
        0.625/0.125/0.250,
        0.625/0.250/0.125,
        0.750/0.125/0.125
    } {
        \pgfmathsetmacro{\px}{\a*\vX + \b*\eoneX + \c*\etwoX}
        \pgfmathsetmacro{\py}{\a*\vY + \b*\eoneY + \c*\etwoY}
        \pgfmathsetmacro{\pz}{\a*\vZ + \b*\eoneZ + \c*\etwoZ}
        
        \draw[->] 
               ({\px-\length*(\px-\vX)},{\py-\length*(\py-\vY)},{\pz-\length*(\pz-\vZ)}) 
            -- ({\px+\length*(\px-\vX)},{\py+\length*(\py-\vY)},{\pz+\length*(\pz-\vZ)});
        
        \pgfmathsetmacro{\qx}{\a*\uX + \b*\eoneX + \c*\etwoX}
        \pgfmathsetmacro{\qy}{\a*\uY + \b*\eoneY + \c*\etwoY}
        \pgfmathsetmacro{\qz}{\a*\uZ + \b*\eoneZ + \c*\etwoZ}
        
        \draw[->] 
               ({\qx-\length*(\uX-\qx)},{\qy-\length*(\uY-\qy)},{\qz-\length*(\uZ-\qz)})
            -- ({\qx+\length*(\uX-\qx)},{\qy+\length*(\uY-\qy)},{\qz+\length*(\uZ-\qz)});
    }
    
\end{tikzpicture}
\caption{
\label{fig: rwg}
The vector field defined by an RWG basis function on two adjacent triangles $t_+$ and $t_-$.}
\end{subfigure}

\caption{An illustration of the relationship between a mesh (a) and an individual RWG basis function (b).}
\end{figure}

We denote by $B_\M$ the RWG basis set associated with mesh $\M$.
Using the machinery from \Cref{section: bem}, we can then speak of the linear system associated with $\M$:
\[
    Z_\M I = V_\M, 
\]
where $Z_\M$ and $V_\M$ are obtained from \Cref{eq: Z definition,eq: V definition} by using the basis functions $f \in B_\M$.
Similarly, given $I$ satisfying the above, the electric field generated by $I$ in the direction $(\vartheta, \varphi)$ is
\[
    c_\M(\vartheta, \varphi)^T I, 
\]
where the coefficients $c_\M(\vartheta, \varphi)$ are obtained from \Cref{eq: far field coefficients} using the basis functions $f \in B_\M$.

Finally, since we aim to optimize the shape of an antenna, it is useful to be able to speak of a sub-mesh of a mesh.

\begin{definition}
    We say that a mesh $\M' = (\V', T')$ is a sub-mesh of $\M = (\V, T)$, written $\M' \subseteq \M$, if $\V' \subseteq \V$ and $T' \subseteq T$.
\end{definition}

\section{Formulating Antenna Design as Mathematical Program}
\label{section: formulation}
We will first discuss some useful facts and new results that are crucial to formulate our antenna design problem as a MISOCP. Subsequently we enhance that basic formulation to eventually make it computationally tractable.
\subsection{The Crucial Trick: A Charnes--Cooper Theorem for Complex Variables} 
\label{section: charnescooper}

This section introduces one of the core ideas that enables the overall approach taken in this paper, namely, to linearize the objective function which by default would be the maximization of the absolute value of a complex number, and hence not tractable.  

For given $A \in \C^{m\times n}$, $b\in \C^m$, we would like to be able to solve programs of the form

\begin{maxi*}|s|
    {}{|c^Tx|}{}{}\label{eq: program P}\tag{P}
    \addConstraint{Ax}{= b}{}
    \addConstraint{x}{\in \C}{}
\end{maxi*}
Note that the objective of this maximization problem is non-concave, and all variables and coefficients are complex-valued. 
Our goal, however, is a tractable mathematical program over the reals. 

The first step, reformulating a complex valued program in the reals is not that difficult; see, e.g.,  \cite{GOEMANS2004442}, where a similar embedding for complex-valued semidefinite programming is used.
To that end, recall that we denote by $\Re$ and $\Im$ the real and imaginary part of a complex number, vector or matrix. Then  we can equivalently represent a linear system $Ax = b$ with $A \in \C^{m\times n}$, $x \in \C^n$, $b \in \C^m$ as as an ordinary linear system over $\R$ via
\[
\begin{pmatrix}
    \Re{A} & -\Im{A}\\
    \Im{A} & \Re{A}
\end{pmatrix}
\begin{pmatrix}
    \Re{x}\\
    \Im{x}
\end{pmatrix}
=
\begin{pmatrix}
    \Re{b}\\
    \Im{b}
\end{pmatrix}\,.
\]

The second step is more challenging, namely how to handle the objective $|c^Tx|$ of \eqref{eq: program P} which is the absolute value of a complex number, hence nonlinear. To fix this issue, the main trick is the following: We introduce an auxiliary complex variable $z$ that shifts the phase of complex vector $b$ (and via $Ax=b$  also $x$) so as to make $\Im c^Tx =0$, so that $c^Tx$ becomes hence real-valued. Note that the necessary phase shift is clearly not known a priori, which is why $z$ needs to be introduced as an auxiliary variable. 

It is this idea that is essential for obtaining a linear objective, making the overall approach viable at all.
Mathematically speaking, this works as long as our objective is to maximize $|c^Tx|$.  

In the context of the physical application, this freedom is also known as \emph{gauge freedom}; see also the remark after the proof of \Cref{thm: base mip} in \Cref{section: basic model}

This idea leads to the following second-order cone program that can be solved instead of \eqref{eq: program P}:

\begin{maxi*}|s|
    {}{\Re c^T y}{}{}
\label{eq: program Q}
\tag{Q}
    \addConstraint{Ay}{= bz}{}
    \addConstraint{|z|}{\leq 1}{}
    \addConstraint{y}{\in \C^n}{}
    \addConstraint{z}{\in \C.}{}
\end{maxi*}

The following theorem shows that programs \eqref{eq: program P} and \eqref{eq: program Q} are equivalent in the sense that their optimal solutions correspond to each other. Before formulating the theorem, we remark that it is in some sense similar to the Charnes--Cooper transformation for linear fractional programming~\cite{charnes1962programming}. In the Charnes--Cooper transform, the input data is scaled so that the numerator of a fractional objective is one. Here, we shift the phase of the complex-valued input data so that the imaginary part of the objective is zero. 

\begin{theorem}
\label{thm: phase fixing trick}
    Let $(y^*, z^*)$ be an optimal solution of \eqref{eq: program Q}.
    
    \begin{enumerate}
        \item If $|z^*| =1$, then $x^* = y^*/z^*$ is an optimal solution to \eqref{eq: program P} and its objective value equals that of $(y^*, z^*)$ for \eqref{eq: program Q}. That is, 
        \[
        x^* \in \arg \max_{x\in P} |c^Tx| \text{ and } \Re c^Ty^* = |c^Tx^*|\,.
        \]
        \item If $|z^*| < 1$, then either the solution can be scaled to an optimal solution $(x', z')$ with $|z'| = 1$, or \eqref{eq: program P} is infeasible or has optimal objective value 0.
        \item If \eqref{eq: program Q} is unbounded, then \eqref{eq: program P} is unbounded or infeasible.
    \end{enumerate}

    \begin{proof}
        Let $(y^*, z^*)$ be an optimal solution of \eqref{eq: program Q} with $|z^*| = 1$. Then $x^* = y^*/z^*$ is feasible for \eqref{eq: program P}, since 
        \[
        Ax^* = Ay^*/z^* = bz^*/z^* = b,
        \]
        and its objective value is 
        \[
        |c^Tx^*| = |c^Ty^*/z^*| = \Re c^Ty^*,
        \]
        where the final equality follows from $z^*$ being chosen to maximize $c^Ty^*$, hence $\Re c^Ty^* = |c^Ty^*|$ follows from optimality of $(y^*, z^*)$.

        Suppose $x^*$ is not optimal for \eqref{eq: program P}, that is, there exists an $x'$ s.t. $Ax' = b$, $|c^Tx'| > |c^Tx^*|$. 
        Then let 
        \begin{align*}
            z' = \frac{\overline{c^Tx'}}{|c^Tx'|} && y' = x'z'\,.
        \end{align*}
        Clearly $(y', z')$ is feasible for \eqref{eq: program P}, and
        \begin{align*}
        \Re c^Ty' &= \Re \left\{c^Tx'\frac{\overline{c^Tx'}}{|c^Tx'|}\right\} = \Re \left\{\frac{|c^Tx'|^2}{|c^Tx'|}\right\} = |c^Tx'|\\ &> |c^Tx^*| = |c^Ty^*/z^*| = |c^Ty^*| \geq \Re c^Ty^*,
        \end{align*}
        contradicting optimality of $(y^*, z^*)$. So $x^*$ is optimal for \eqref{eq: program P}.

        Now let $|z^*| < 1$ in an optimal solution $(y^*, z^*)$ of \eqref{eq: program Q}. Then either $|z^*| > 0$, and $(y^*/|z^*|, z^*/|z^*|)$ is also a solution with objective 
        \[
        \Re c^Ty^*/|z^*| \geq \Re c^Ty^*,
        \] 
        with inequality if $\Re c^Ty^* >0$, contradicting optimality of $(y^*, z^*)$, and equality if $\Re c^Ty^* = 0$, in which case $(y^*/|z^*|, z^*/|z^*|)$ is also optimal with $|z^*/|z^*|| = 1$.

        The remaining case is when $z^* = 0$ in an optimal solution of \eqref{eq: program Q}. But then $(y^*, z^*)$ is such that $Ay^* = 0$.
        If $\Re c^Ty^* > 0$, then \eqref{eq: program Q} is unbounded, contradicting optimality of $(y^*, z^*)$.
        
        Therefore, $\Re c^Ty^* = 0$. If there existed a solution $x'$ to \eqref{eq: program P} with $|c^Tx'| > 0$, then by similar reasoning as above we could find a solution $(y', z')$ with $\Re c^Ty' >0$, contradicting optimality of $(y^*, z^*)$. So the optimal objective value of \eqref{eq: program P} is 0, or \eqref{eq: program P} is infeasible.

        Now, assume that \eqref{eq: program Q} is unbounded. Then there must exist a $y'$ such that $Ay' = 0$ and $\Re c^Ty' > 0$. If there also exists an $x'$ such that $Ax' = b$, then \eqref{eq: program P} would be unbounded, and if not, it is infeasible.
    \end{proof}
\end{theorem}
Note that $(0, 0)$ is always a solution to \eqref{eq: program Q}, so that it cannot be infeasible. Note also that in any optimal solution $(y^*, z^*)$ to \eqref{eq: program Q}, 
\begin{align*}
    \Im c^Ty^* = 0, \Re c^Ty^* \geq 0, 
\end{align*}
which can be added as constraints to \eqref{eq: program Q} to obtain a stronger formulation without invalidating the above proof.

If one needed to deduce feasibility of \eqref{eq: program P}, one could always attempt to solve the following linear program, feasibility of which directly implies feasibility of \eqref{eq: program P}:
\begin{maxi*}|s|
    {}{0}{}{}
    \addConstraint{Ax}{=b}{}
    \addConstraint{x}{\in \C^n.}{}
\end{maxi*}

We can also combine our technique with that of Charnes--Cooper to cover maximizing the absolute value of a fraction of linear expressions in complex variables, which is then a generalization of Charnes--Cooper. 
The proof of the following theorem is deferred to the appendix.
\begin{restatable}{theorem}{charnescooper}
\label{thm:charnescooper}
For a given matrix $A \in \C^{n\times m}$, vectors $c, d \in \C^m$ and scalars $\alpha, \beta \in \C$, let \eqref{eq: problem R} and \eqref{eq: problem S} be as follows:
    \begin{maxi*}|s|
        {}{\left|\frac{c^Tx + \alpha}{d^Tx + \beta}\right|}{}{}
            \label{eq: problem R}
            \tag{R} 
        \addConstraint{Ax}{=b}{}
        \addConstraint{x}{\in \C^m}
    \end{maxi*}
    \begin{maxi*}|s|
        {}{\Re\{c^T y + \alpha z\}}{}{}
            \label{eq: problem S}
            \tag{S}
        \addConstraint{Ay}{=bz}{}
        \addConstraint{|d^T y + \beta z|}{\leq 1}{}
        \addConstraint{y}{\in \C^m}{}
        \addConstraint{z}{\in \C.}{}
    \end{maxi*}

Extending the usual definitions of infeasibility and unboundedness to \eqref{eq: problem R}, we say that \eqref{eq: problem R} is \emph{infeasible} if there exists no $x \in \C^m$ such that $Ax = b$ and $d^Tx + \beta \neq 0$, and that \eqref{eq: problem R} is \emph{unbounded} if it has no finite maximum, for instance when there exists an \emph{unbounded ray} $x'$ such that $Ax' = 0$, $c^Tx' + \alpha \neq 0$ and $d^Tx' =0$, and an $\hat{x}$ such that $A\hat{x} = b$ and $d^T\hat{x} + \beta \neq 0$.

Then: 
\begin{enumerate}
    \item If $(y^*, z^*)$ is an optimal solution to \eqref{eq: problem S} with $|d^Ty^* +\beta z^*| = 1$ and $|z^*| > 0$, then $x^* = y^*/z^*$ is an optimal solution to \eqref{eq: problem R} and 
    \[
    \Re\{c^Ty^* + \alpha z^*\} = \left|\frac{c^Tx^* + \alpha}{d^Tx^* + \beta}\right|,
    \]
    \item If $(y^*, z^*)$ is an optimal solution to \eqref{eq: problem S} with $|d^Ty^* +\beta z^*| < 1$ or $z^* = 0$, then $(y^*, z^*)$ can be scaled to an optimal solution $(y', z')$ to \eqref{eq: problem S} with $|d^Ty' +\beta z'| = 1$, or \eqref{eq: problem R} is infeasible or has optimal objective value 0, and
    \item If \eqref{eq: problem S} is unbounded, then \eqref{eq: problem R} is infeasible or unbounded.
\end{enumerate}
\end{restatable}
To streamline the main body of the paper, we provide the proof in Appendix~\ref{app:charnescooper}.

\subsection{The Natural Formulation}
\label{section: basic model}

We now derive an exact MISOCP for antenna topology optimization.
The input to this MISOCP is a mesh $\M$, and an optimal solution is a sub-mesh $\M' \subseteq \M$ that maximizes the magnitude of a given component of the radiated electric field in a given direction; that is, $|e_i \cdot \EFF(\vartheta, \varphi)|$.
Using the machinery from \Cref{section: electromagnetism,section: bem,section: discrete geometry}, we can write this antenna design problem as 

\begin{maxi*}|s|
    {\scriptstyle \M' \subseteq \M}{|c_{\M'}^T I|}{}{}
        \label{eq: natural formulation}
        \tag{N}
    \addConstraint{Z_{\M'}I}{=V_{\M'}}{}
    \addConstraint{I}{\in \C^{B_{\M'}}.}{}
\end{maxi*}

While this is a natural formulation of the mesh-selection problem outlined at the end of \Cref{section: discrete geometry}, it is not in the form of a typical mathematical program.
In particular, we note that the dimension of the variable vector $I$ depends on the feasible solution $\M' \subseteq \M$.
It takes a some work to reformulate this program in a more familiar way.
Moreover, it is important to recall that all numerical values are complex-valued, hence the absolute value $|\cdot |$ is the square root of a quadratic expression in the real and imaginary parts of a complex number. Therefore, our main result in this section is that this program can be formulated as a nearly-linear MIP,  that is, a linear MIP with only a single convex quadratic constraint.
\begin{theorem}\label{thm: base mip}
    Fix a mesh $\M = (\V, T)$, and let $Z = Z_\M$ and $V = V_\M$.
    The optimal objective value of the antenna design problem  
    \[
    \max \{|c_{\M'}^T I| \, : \, \M'\subseteq \M, \, Z_{\M'}I = V_{\M'}, \, I \in \complex^{B_{\M'}}\}
    \]
    is equal to the optimal objective value of
    \begin{maxi*}|s|
        {x \in \{0, 1\}^{T(\M)}}{\Re\{c_\M^T I\}}{}{}
        \addConstraint{y_f = 1}{\implies Z_{f, *}I = V_f \cdot z}{\quad \forall f \in B_\M}
        \addConstraint{y_f = 0}{\implies I_f = 0}{\quad \forall f \in B_\M}
        \addConstraint{\prod_{t \in \supp(f)} x_t}{= y_f}{\quad \forall f \in B_\M}
        \addConstraint{|z|}{\leq 1}{\quad z \in \complex.}
    \end{maxi*}
\end{theorem}

The proof of \Cref{thm: base mip} follows from \Cref{lem: base mip constraints} in the next section.

\subsection{Almost-Linear MIP for Antenna Optimization}

We write the mesh-selection problem in a more standard MIP form, by introducing binary variables to represent the triangles of the mesh.
This requires some care, since the matrix $Z_\M$ is indexed not by $T(\M)$ but by the basis functions $B_\M$; thus, the presence or absence of a basis function $f \in B_{\M'} \subseteq B_\M$ must be linked to the binary variables $x_t$ for $t \in \supp(f)$.
\Cref{lem: base mip constraints} summarizes this for an arbitrary $\M$.

\begin{lemma}\label{lem: base mip constraints}
    Let $\M = (\V, T)$ be a mesh.
    Then the feasible region of the antenna optimization problem (\ref{eq: natural formulation})
    \[
        P_0 = \{(\M', I) \, : \, 
            \M' \subseteq \M, \, 
            Z_{\M'} I = V_{\M'}, \, 
            I \in \complex^{\M'}
            \}
    \]
    is isomorphic to
    \begin{align*}
        P_1 = \left\{(x, I) \in \{0, 1\}^{T} \times \complex^{B_\M}\, : \,
            \left(\prod_{t \in \supp(f)} x_t = 1 \implies Z_{f, *} I = V_f\right),\right.\\ 
            \left.\left(\prod_{t \in \supp(f)} x_t = 0 \implies I_f = 0\right)\right\}\,.
    \end{align*}
\end{lemma}

\begin{proof}
    We identify an isomorphism $\phi : P_0 \to P_1$.
    A mesh $\M' \subseteq \M$ is mapped to a binary vector $x \in \{0, 1\}^T$ by setting $x_t = 1$ if and only if $t \in T(\M')$; that is, $x$ is the incidence vector of the sub-mesh $\M'$ of $\M$.
    A vector $I \in \complex^{\M'}$ is mapped to $\tilde I \in \complex^M$ by setting 
    \[
        \tilde I_f = \begin{cases} 
            I_f & \text{if $f \in B_{\M'}$}, \\
            0 & \text{otherwise}.
        \end{cases}
    \]
    We define the inverse map $\phi^{-1}$ in the obvious way: For $(x, \tilde I) \in P_1$, let $(\M', I) = \phi^{-1}(x, \tilde I)$ be given by
    \[
        \M' = (\V, \{t \in T(\M) \, : \, x_t = 1\})
    \]
    and $I_f = \tilde I_f$ for each $f \in B_{\M'}$.
    Let $(\M', I) \in P_0$.
    Then $Z_{\M'} I = V_{\M'}$, where for each $f_i, f_j \in B_{\M'}$
    \[
        Z_{\M', f_i f_j} = \left\langle \mathcal{L} f_i, f_j\right\rangle, 
        \quad
        V_{\M', f_j} = \langle n \times E^{\text{inc}}, f_j \rangle.
    \]
    We must show that $(x, \tilde{I}) = \phi(\M', I) \in P_1$.

    Let $f \notin B_{\M'}$.
    We have
    \[
        \prod_{t \in \supp{f}} x_t = 0,
    \]
    since $f \notin B_{\M'}$ implies that for some $t \in \supp(f)$ we have $t \notin T(\M')$.
    Moreover, $\tilde{I}_f = 0$ by construction.

    On the other hand, for $f \in B_{\M'}$, we have $\prod_{t \in \supp(f)} x_t = 1$, and
    \begin{align}\label{eq: Z constraints}
        Z_{\M, f, *}\tilde I &= \sum_{f_j \in B_\M} Z_{\M, f, f_j} \tilde I_{f_j}\notag\\
            &= \sum_{f_j \in B_{M'}} Z_{\M, f, f_j} \tilde I_{f_j}
                + \sum_{f_j \in B_\M \setminus B_{\M'}} Z_{\M, f, f_j} \underbrace{\tilde I_{f_j}}_{= 0}\\
            &= \sum_{f_j \in B_{M'}} Z_{\M, f, f_j} \tilde I_{f_j} \nonumber 
            = Z_{\M', f, *} I 
            = V_f.
    \end{align}
    This shows that $(x, \tilde I) \in P_1$ as desired.
    For the other direction, we must show that $\phi^{-1}(x, \tilde I) = (\M', I) \in P_0$, which follows immediately from \Cref{eq: Z constraints}.
\end{proof}

Now that we have a more standard mathematical program for the antenna optimization problem, we apply \Cref{thm: phase fixing trick} to prove \Cref{thm: base mip}.

\begin{proof}[Proof of \Cref{thm: base mip}]
    To begin, we use \Cref{lem: base mip constraints} to transform the formulation (\ref{eq: natural formulation}) into a more standard mathematical program.
    To that end, we transform the feasible set using the isomorphism $\phi$ from \Cref{lem: base mip constraints}, and write the program equivalently as a bilevel optimization problem,
    \begin{maxi*}|s|
        {\scriptstyle x \in \{0, 1\}^{T(\M)}}{\max_{I \in \C^{B_\M}} |c_\M^T I|}{}{}
        \addConstraint{y_f = 1}{\implies Z_{f, *}I = V_f}{\quad \forall f \in B_\M}
        \addConstraint{y_f = 0}{\implies I_f = 0}{\quad \forall f \in B_\M}
        \addConstraint{\prod_{t \in \supp(f)} x_t}{= y_f}{\quad \forall f \in B_\M}.
    \end{maxi*}
    The constraints $\prod_{t \in \supp(f)} x_t = y_f$ can be linearized straightforwardly, since the $x_t$ variables are each binary.
    Since the variables $x \in \{0, 1\}^{T(\M)}$ are fixed in the inner problem, this problem can be written in the form $\max |c^T y| \text{ s.t. } Ay = b$.
    We can then use \Cref{thm: phase fixing trick} to write the inner problem in the equivalent form
    \begin{maxi*}|s|
        {\scriptstyle I \in \C^{B_\M}}{\Re \{c_\M^T I\}}{}{}
        \addConstraint{y_f = 1}{\implies Z_{f, *}I = V_f\cdot z}{\quad \forall f \in B_\M}
        \addConstraint{y_f = 0}{\implies I_f = 0}{\quad \forall f \in B_\M}
        \addConstraint{\prod_{t \in \supp(f)} x_t}{= y_f}{\quad \forall f \in B_\M}
        \addConstraint{|z|}{\leq 1}{\quad z \in \C.}
    \end{maxi*}
    Writing the problem as a single-level optimization problem over $x$ again completes the proof.
\end{proof}

\begin{remark}
    \label{remark:gauge freedom}
    In the case of antenna optimization, the reformulation in this proof has a physical interpretation.
    If $J$ is a current distribution on the surface of an antenna that produces the far field $E_\mathrm{FF}$, then the distribution $J' = e^{\jim \phi}J$ produces the field $E_\mathrm{FF}' = e^{\jim \phi} E_\mathrm{FF}$.
    For the radiation characteristics of the antenna, this change is immaterial, since the relevant quantities are independent of this phase change.
    The change $J \to e^{\jim \phi} J$ can therefore be seen as a \emph{gauge transformation}~\cite{griffithsIntroductionElectrodynamics2023}, and adding the variable $z \in \complex$ with $|z| = 1$ to impose $\Im\{\EFF\} = 0$ is a form of \emph{gauge fixing}.
\end{remark}

\subsection{Additional constraints}

The MIP derived in \Cref{thm: base mip} is an exact formulation for finding the sub-mesh $\M' \subseteq \M$ that maximizes the far-field amplitude in a fixed direction, subject to a fixed incident electric field.
In practical cases however, various other design considerations must be taken into account to produce a useful antenna design.
We outline two such design considerations: Geometric and impedance constraints.
Geometric constraints are easily realizable as linear constraints.
To constrain impedance, we require a single additional second-order cone constraint.

\subsubsection{Connectivity constraints}
\label{sec: connectivity}
\Cref{thm: base mip} has in its feasible region all sub-meshes $\M' \subseteq \M$, providing a great deal of flexibility.
In some cases, this is not desirable, for instance, when an optimal solution consists of many disjoint components.
Such topologies may be difficult to manufacture, and can moreover hinder the search for a good MIP solution.
There may also be geometric configurations that one wishes to exclude, such as disconnecting the feed point from the rest of the mesh.

Any geometric constraint that can be efficiently encoded as a graph topology constraint is easily incorporated into \Cref{thm: base mip}.
Given a mesh $\M = (T, \V)$, we construct the graph $G_\M$ with vertex set $V(G_\M) = T$ and edge set
\[
    E(G_\M) = \{\{t_i, t_j\} : \text{$t_i$ and $t_j$ share an edge in $\M$}\}.
\]
Selecting a sub-mesh $\M' \subseteq \M$ is equivalent to selecting a subset of $V(G_\M)$, allowing an easy translation of topological mesh constraints to graph topology constraints.

As a concrete example, we consider connectivity constraints: We only wish to select meshes $\M'$ for which $G_{\M'}$ is connected.
Since the feed region of an antenna is surely selected in any optimal solution, we can immediately say that any triangles touching the feed region must be selected.
Collecting these triangles into a set $F \subseteq T$, we have
\[
    x_t = 1, \text{ for each $t \in F$}.
\]
Then we can apply a flow formulation for the connectivity constraint.
For every edge $\{t_i, t_j\} \in E(G_\M)$, we define a forward and backward flow variable $f_{ij},f_{ji}\in \real_{\geq 0}$.
We associate with every $t \in T \setminus F$ a source $x_t$, and enforce flow conservation at every triangle $t \in T \setminus F$:
\[
    \sum_{i \in N(j)} (f_{ji} -f_{ij}) = x_j, \text{ for each } j \in T \setminus F
\]
We also enforce
\begin{align*}
    f_{ij} \leq M x_{i} &\text{ for each $i \in T$}, j \in N(i),\\
    f_{ij} \leq M x_{j} &\text{ for each $i \in T$}, j \in N(i),
\end{align*}
which couples the flow to the presence of the triangles between which it runs.

Triangles in $F$ act as sinks; we enforce there
\begin{align*}
    \sum_{i \in N(j)} f_{ji} &= 0, &\text{ for each $j \in F$},\\
    \sum_{j\in F}\sum_{i \in N(j)}f_{ij} &=\sum_{k\in T \setminus F} x_k, &\text{ for each $j \in F$}.
\end{align*}
The first constraint enforces that no flow exits a feed triangle (which helps constrain the number of symmetric solutions in the case of multiple connected feed triangles), and the second constraint ensures flow conservation by implying
\[
0\leq\sum_{i \in N(j)} f_{ij} \leq \sum_{k\in T \setminus F} x_k, \text{ for each $j \in F$},\\
\]
and ensuring that the total outflow equals the total inflow.
These constraints are sufficient to ensure that the selected triangles form a connected region.

\subsubsection{Symmetry}
\label{section: symmetry}
For certain types of antennas (like the ones we will consider in \Cref{section: results}), it stands to reason that there exists an optimal symmetric design, and that many nearly optimal, slightly asymmetric designs also exist.

We can therefore significantly reduce the complexity of optimization by enforcing symmetry in these designs - both in the design variables $x_j, j \in T$ and the current coefficients $I \in \C^{B_\M}$ (as the current distributions on these antennas will then also be symmetric, if they are fed from the center).
The principle is that one chooses one side of the antenna, and expresses the current coefficients for one side in terms of those for the other side. One can then compress the BEM system, so that it is written purely in terms of the variables corresponding to one side of the antenna; the optimization process then proceeds with approximately half the number of variables.
\subsubsection{Impedance constraints}
\label{sec: impedance}
One of the most critical characteristics of an antenna is its \emph{input impedance}.
This is defined simply as the ratio of the antennas driving voltage to the (complex) current it draws:
\begin{align}
\label{eq: impedance definition}
    Z_\mathrm{in} := \frac{V_\mathrm{in}}{I_\mathrm{feed}}.
\end{align}
When $V_\mathrm{in}$ is measured in Volts ($\mathrm{V}$) and the current in Amperes ($\mathrm{A}$), the impedance is measured in Ohms ($\Omega$).
The feed current $I_\mathrm{feed}$ is computed as
\[
    I_\mathrm{feed} = \int_{\Gamma_\mathrm{feed}} J(x) \cdot n(x)\dd{x},
\]
where $\Gamma_\mathrm{feed}$ is a geometric line across the feed region of the antenna, and $n$ is a unit normal vector to $\Gamma_\mathrm{feed}$ pointing into the antenna.
That is, $I_\mathrm{feed}$ is the total current crossing $\Gamma_\mathrm{feed}$.
Inserting $J = \sum_{f \in B_\M} I_f f$,
\[
    I_\mathrm{feed} = \sum_{f \in B_\M} I_f \int_{\Gamma_\mathrm{feed}} f(x) \cdot n(x)\dd{x} = c_\mathrm{feed}^T I,
\]
with $c_{\mathrm{feed}, f} = \int_{\Gamma_\mathrm{feed}} f(x) \cdot n(x)\dd{x}$.
Thus, $I_\mathrm{feed}$ is a (complex) linear form in the current coefficients $I$.

The antenna is connected to its power source via a \emph{feed line}, which itself has a characteristic impedance $Z_0$; typically $Z_0 = 50 \, \Omega$.
If $Z_\mathrm{in}$ is very near $Z_0$, then the antenna is \emph{well-matched}.
This implies that most of the power delivered to the antenna is accepted by the antenna, and (if the material is electrically lossless) eventually radiated.
Conversely, if the antenna is badly matched (so $Z_\mathrm{in} \not\approx Z_0$), then much of the delivered power is \emph{reflected} back to the feed line, where it is eventually dissipated as heat.

This has practical consequences for the program of \Cref{thm: base mip}.
Even if we obtain an optimal solution with a good objective value--that is, with a high-amplitude radiated field--this may come at the cost of a very inefficient antenna.
In a physical antenna this high field amplitude would be achieved simply by drawing a large amount of energy from the power source, but reflecting most of it instead of radiating it out.
This effect is not taken into account in the model described in \Cref{section: prelims}, since this model only sees the current accepted by the antenna.
Therefore, any antenna optimization model is incomplete without incorporating a constraint on the impedance.

Therefore, we introduce the additional second-order cone constraint
\begin{align}
\label{eq: impedance lower bound}    
    |I_\mathrm{feed}| \leq \frac{V_\mathrm{in}}{|Z_0|},
\end{align}
which guarantees that
\[
|Z_\mathrm{in}| \geq |Z_0|.
\]
This constraint has the limitation that it does not separately constrain the real and imaginary parts of the impedance; therefore, $Z_\mathrm{in} = 50 \, \Omega$, $Z_\mathrm{in} = 50\jim \Omega$ and $Z_\mathrm{in} = \sqrt{50} + \sqrt{50} \jim \, \Omega$ all satisfy this constraint, while only the first is desirable.
In many of our test cases, however, we have found that this constraint tends to yield optimal designs with $Z_\mathrm{in} \approx Z_0 = 50 \, \Omega$, without the need to set explicit constraints on the imaginary part of the impedance.

In our second paper~\cite{hoeksma2026physics}, we introduce an efficient method of constraining the real and imaginary parts of the input impedance separately.

\section{Results}
\label{section: results}

We will first validate our method by searching for the optimal length of a dipole antenna, and comparing to known literature values. Then, we will optimize a so-called Yagi-Uda antenna.
For both cases, the base mesh $\M$ was generated using the open-source program Gmsh~\cite{geuzaine2009gmsh}.
The electromagnetics simulations necessary to obtain the matrices $Z_M$, were performed using the open-source boundary element method library Bempp-cl~\cite{betcke2021bempp}.
All optimization was performed using Gurobi 11.0.2~\cite{gurobi} on a  system with  32 Intel Xeon Gold 5217 cores and 16 GB of memory.

\subsection{Dipole Antenna}
\label{section: dipole}

We first optimize a microstrip dipole antenna (see \Cref{fig: dipole mesh}). 
This is a simple antenna, consisting of a small strip of perfect electrical conductor. 
The strip has a voltage applied to its center, so that a standing wave of electrical current forms on the antenna. 
The dipole antenna is a classical antenna design, and as such has been analyzed also analytically.

In these analyses, a typical approximation is that the current runs solely along the length of the antenna; this is justifiable as the width of the antenna is ``small'' compared to its length.
Under such an assumption, one can compute the length that maximizes the radiated field amplitude perpendicular to the antenna's long axis; literature values depend on the exact width, but typically the optimal length  is between $0.46\lambda$ and $0.48\lambda$, where $\lambda = 2\pi c/\omega$ is the \emph{wavelength} corresponding to the frequency at which we feed the antenna~\cite{balanisAntennaTheoryAnalysis2012}.

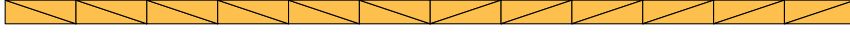
\begin{figure}
    \centering
    
    \begin{tikzpicture}
    [
      y={(0cm,1cm)},
      x={(1cm,0cm)},
      z={(0cm,1cm)},
      scale=2.5
    ]
    \tikzset{stlfacet/.style={fill=pec, draw=black, line join=round}}
        \path[stlfacet] (-1.8737,-0.0625,0) -- (-2.2484,0.0625,0) -- (-2.2484,-0.0625,0) -- cycle;
\path[stlfacet] (-1.8737,-0.0625,0) -- (-1.8737,0.0625,0) -- (-2.2484,0.0625,0) -- cycle;
\path[stlfacet] (-1.499,-0.0625,0) -- (-1.8737,0.0625,0) -- (-1.8737,-0.0625,0) -- cycle;
\path[stlfacet] (-1.499,-0.0625,0) -- (-1.499,0.0625,0) -- (-1.8737,0.0625,0) -- cycle;
\path[stlfacet] (-1.1242,-0.0625,0) -- (-1.499,0.0625,0) -- (-1.499,-0.0625,0) -- cycle;
\path[stlfacet] (-1.1242,-0.0625,0) -- (-1.1242,0.0625,0) -- (-1.499,0.0625,0) -- cycle;
\path[stlfacet] (-0.7495,-0.0625,0) -- (-1.1242,0.0625,0) -- (-1.1242,-0.0625,0) -- cycle;
\path[stlfacet] (-0.7495,-0.0625,0) -- (-0.7495,0.0625,0) -- (-1.1242,0.0625,0) -- cycle;
\path[stlfacet] (-0.3747,-0.0625,0) -- (-0.7495,0.0625,0) -- (-0.7495,-0.0625,0) -- cycle;
\path[stlfacet] (-0.3747,-0.0625,0) -- (-0.3747,0.0625,0) -- (-0.7495,0.0625,0) -- cycle;
\path[stlfacet] (0,-0.0625,0) -- (-0.3747,0.0625,0) -- (-0.3747,-0.0625,0) -- cycle;
\path[stlfacet] (0,-0.0625,0) -- (0,0.0625,0) -- (-0.3747,0.0625,0) -- cycle;
\path[stlfacet] (0,-0.0625,0) -- (0.3747,0.0625,0) -- (0,0.0625,0) -- cycle;
\path[stlfacet] (0,-0.0625,0) -- (0.3747,-0.0625,0) -- (0.3747,0.0625,0) -- cycle;
\path[stlfacet] (0.3747,-0.0625,0) -- (0.7495,0.0625,0) -- (0.3747,0.0625,0) -- cycle;
\path[stlfacet] (0.3747,-0.0625,0) -- (0.7495,-0.0625,0) -- (0.7495,0.0625,0) -- cycle;
\path[stlfacet] (0.7495,-0.0625,0) -- (1.1242,0.0625,0) -- (0.7495,0.0625,0) -- cycle;
\path[stlfacet] (0.7495,-0.0625,0) -- (1.1242,-0.0625,0) -- (1.1242,0.0625,0) -- cycle;
\path[stlfacet] (1.1242,-0.0625,0) -- (1.499,0.0625,0) -- (1.1242,0.0625,0) -- cycle;
\path[stlfacet] (1.1242,-0.0625,0) -- (1.499,-0.0625,0) -- (1.499,0.0625,0) -- cycle;
\path[stlfacet] (1.499,-0.0625,0) -- (1.8737,0.0625,0) -- (1.499,0.0625,0) -- cycle;
\path[stlfacet] (1.499,-0.0625,0) -- (1.8737,-0.0625,0) -- (1.8737,0.0625,0) -- cycle;
\path[stlfacet] (1.8737,-0.0625,0) -- (2.2484,0.0625,0) -- (1.8737,0.0625,0) -- cycle;
\path[stlfacet] (1.8737,-0.0625,0) -- (2.2484,-0.0625,0) -- (2.2484,0.0625,0) -- cycle;
    \end{tikzpicture}
    
    \caption{An example of a mesh for a microstrip dipole antenna.}
    \label{fig: dipole mesh}
    
\end{figure}

Such dipole antennas radiate a field with a characteristic toroidal shape.
It is therefore a good test case for our method whether we can (i) recover the known optimal length, and (ii) recover the correct shape of this radiated field.
We validated our method by modeling a dipole of length $1.5\lambda$ following \Cref{section: prelims,section: bem}.

\paragraph*{Mesh.}~
We discretized the strip dipole using a mesh (cf.\ \Cref{section: discrete geometry}) of $n = 320$ mesh elements
The mesh has a width of $W = 0.00625\lambda$ and a length of $L = 3\lambda/2$
(The precise width is not very important for the radiation characteristics of a dipole, so long as $W \ll L$~\cite{balanisAntennaTheoryAnalysis2012}.)

\paragraph*{Constraints.}~
Following \Cref{sec: impedance}, we impose the constraint $|Z_{\mathrm{in}}| \geq 50 \Omega$.

We moreover impose three geometric constraints.
First, since the dipole antenna is a connected strip of metal, we impose the connectivity constraints of \Cref{sec: connectivity}.
Second, strip dipoles are rectangular in shape, which means that removing triangular mesh elements may create undesirable geometries.
We therefore partition the mesh elements into $N = 160$ rectangular patches, constraining these patches so that both all triangles of a patch are enabled or neither is.
Third, we impose that the solutions are symmetric, following \Cref{section: symmetry}.

\paragraph*{Results.}~
Our computations reconfirm that the optimal sub-mesh should have a length of $0.47\lambda$, which agrees with the known range for the optimal length of a dipole antenna.
The computation time was 1133 seconds.
Due to the coarse discretization, the optimal length can only be determined up to an error of $3\lambda/(2N) \approx 0.0094\lambda$.

\begin{figure}
    \centering
    \includegraphics[width=\linewidth]{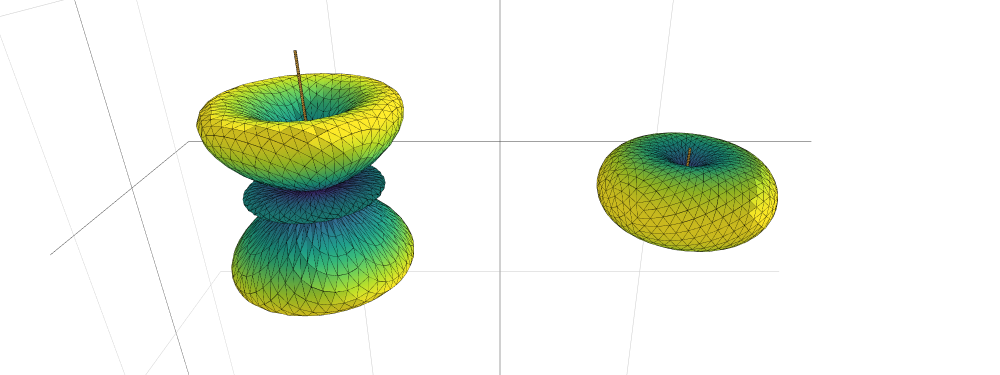}
    \caption{Initial and optimized radiation pattern of the microstrip dipole antenna test case, for $N = 160$ rectangular patches.
        The radial distance of a point on the surface from the origin corresponds the field strength in that direction.
    }
    \label{fig: dipole results}
\end{figure}

\subsection{Yagi-Uda Antenna}
\label{section: yagi uda}

The classic Yagi--Uda antenna~\cite{yagi1926projector} consists of a dipole together with specially-spaced \emph{parasitic elements}---pieces of metal that are not actively fed, but that (ideally) respond to the electromagnetic fields generated by the dipole to generate a more directional antenna pattern. 
Even though this antenna type is rather straightforward and a century old, no generally valid closed-form expression of its antenna pattern is known, nor a closed-form expression for the optimal lengths and positions of the elements. For an extensive treatment of these antennas, see~\cite[Chapter 25]{orfanidisElectromagneticWavesAntennas2016}.

As the Yagi--Uda antenna is relatively simple to model and widely used, it is often used as a test case for optimization methods, e.g., \cite{vanbeurden2003integro}.
Especially parametric gradient-based algorithms are often used, since a Yagi--Uda of $N$ dipole elements can be parameterized by $2N-1$ variables: the $N$ lengths and the $N-1$ spacings between the elements.
Such gradient-based methods converge to some local optimum, and require significant computational resources, since for a fixed mesh $\M$ the matrix $Z_{\M}$ depends non-linearly on the element lengths and spacings.

\begin{figure}
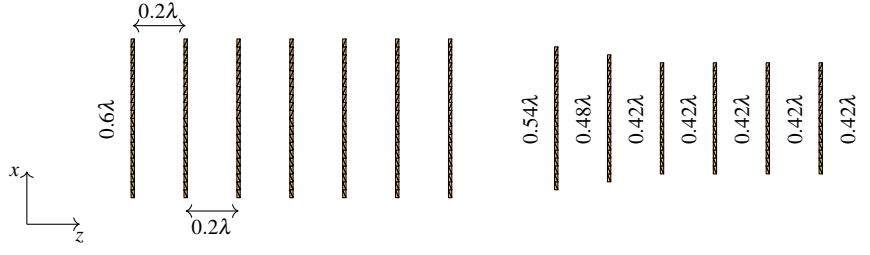

    \centering
        \begin{tikzpicture}
    [
      scale=0.35,
      x={(0cm,1cm)},
      y={(1cm,0cm)},
      z={(0cm,0cm)}
    ]
        \pgfmathsetmacro{\lam}{10}
        \pgfmathsetmacro{\d}{0.5*\lam}
        \pgfmathsetmacro{\offset}{0.5}
        \pgfmathsetmacro{\L}{0.6*\lam}
        
        \tikzset{stlfacet/.style={fill=pec, draw=black, line join=round}}
        \input{yagi_7.tikz}

        \node[rotate=90] () at (0, -0.2*\lam-2*\offset) {$0.6\lambda$};
        
        \draw[<->] (\L/2+\offset, -0.2*\lam) -- (\L/2+\offset, 0) node[midway, above, sloped] {$0.2\lambda$};
        \draw[<->] (-\L/2-\offset, 0) -- (-\L/2-\offset, 0.2*\lam) node[midway, below, sloped] {$0.2\lambda$};

        \draw[->] (-\L/2-2*\offset, -\d-2*\offset) -- ++ (2, 0) node[left] {$x$};
        \draw[->] (-\L/2-2*\offset, -\d-2*\offset) -- ++ (0, 2) node[below] {$z$};

        \tikzset{shift={(0, 1.6*\lam)}}
        \input{yagi_opt_7.tikz}

        \foreach \y/\len in {
          -0.2/0.54,
           0.0/0.48,
           0.2/0.42,
           0.4/0.42,
           0.6/0.42,
           0.8/0.42,
           1.0/0.42
        } {
            \node[rotate=90] () at (0, \y*\lam-2*\offset) {$\len\lambda$};
        }
    \end{tikzpicture}
    \caption{Initial mesh and optimal sub-mesh for the seven-element Yagi--Uda antenna instance.}
    \label{fig: yagi mesh}
\end{figure}

\paragraph{Mesh.}~
Since the Yagi--Uda consists of a set of disjoint dipole antennas, our modeling approach here is similar to the previous test case.
We place $3$ to $7$ dipole meshes with 40 triangles each, of width $W = 0.00625\lambda$ and length $L = 0.6\lambda$ at a fixed inter-element spacing of $\Delta x = 0.2\lambda$. 
See \Cref{fig: yagi mesh}.

\paragraph{Constraints.}~ 
We constrain the length of each dipole to be at least $0.2\lambda$.
These figures are in line with generally recommended guidelines for designing Yagi--Uda antennas; see, e.g., Balanis~\cite{balanisAntennaTheoryAnalysis2012}.
We also constrain each dipole element to be individually a connected component, using a flow formulation similar to the one detailed in \Cref{sec: connectivity}, but adapted to deal with the separate connected components. As with the dipole, we enforce that each element is symmetric.
Following \Cref{sec: impedance}, we enforce 
\[
|Z_{\mathrm{in}}| \geq 50 \, \Omega.
\]

\paragraph{Objective.}~
Special attention must be paid to the objective function.
Since the Yagi--Uda antenna is meant to generate a highly directed beam, it pays to also consider the radiation sent in the reverse direction--the \emph{backlobe radiation}.
Naively maximizing the forward radiation can yield solutions with a large backlobe field, which is undesirable.
We therefore add a penalty term to the objective function of \Cref{thm: base mip}:
\[
    \max_{x \in \{0, 1\}^n} \Re\{c_\M^T I\} - |c_{\M, \mathrm{back}}^T I|.
\]
Here, $c_{\M,\mathrm{back}}$ is derived in the same way as $c_\M$, but yields instead the complex electric field in the negative-$y$ direction.
Since the second term above is convex, this objective function is concave.

Note that this is not the only way in which one can control the backlobe radiation levels; one could also add a (second-order cone) constraint imposing that $|c_{\M, \mathrm{back}}^T I|$ is below some prescribed level, or maximize the so-called \emph{front-to-back ratio}
\[
FBR = \left|\frac{c_\M^T I}{c_{\M, \mathrm{back}}^T I}\right|
\]
by application of \Cref{thm:charnescooper}. In our experimentation, however, we have found that this is generally more costly.

\paragraph{Results.}

We are able to find provably optimal Yagi-Uda designs with up to 7 elements (see \Cref{tab:yagi-scaling}). We see that, as expected, these designs improve with the number of elements. 
In Figures \ref{fig: yagi mesh} and \ref{fig: yagi results}, one can see the difference between the initial and optimized mesh and their associated antenna patterns for the 7-element case.
We also see that the runtime scales very poorly with the number of elements; where the 3-element case is solved in just 7 seconds, the 7-element case takes 42575 seconds.
We believe that there are several contributing factors:

\begin{enumerate}
    \item Because we lack bounds on the current variables $I$ of Program \eqref{eq: natural formulation}, the solver is unable to provide us with reasonable dual bounds. This problem is addressed in a sequel paper~\cite{hoeksma2026physics}.
    
    \item The matrix $Z_\M$ has a very high condition number, which increases with the size of the matrix. This is a known consequence of fundamental properties of the operator $\mathcal{L}$, whose singular values accumulate around zero and infinity~\cite{Andriulli20082398, Ndlec2001}. Preconditioners for $Z_\M$ do exist~\cite{Andriulli20082398}, but a naive application of these destroys the sparsity and condition number of other blocks of the constraint matrix.
    
    \item The matrix $Z_\M$ is large and dense, whereas most MIP solvers are optimized for hypersparse problems, as these are much more typical. Therefore, significant additional solver overhead might be incurred.
\end{enumerate}

All in all, this test case shows that our method can compute provably optimal antennas within the given (discrete) design space, also for non-trivial design problems.
However, it also makes clear that there is still room for improvement, particularly with respect to the above points.

\begin{table}
    \centering
    \caption{Running time scaling for Yagi-Uda antennas with 3 to 7 elements, meshed with 40 triangles each.}
    \label{tab:yagi-scaling}
    \begin{tabular}{lccccc}\toprule
         Antenna elements       &  3    &  4    &  5    &  6    &  7     \\\midrule
         Mesh triangles         & 120   & 160   & 200   & 240   & 280    \\
         Time to optimality (s) & 7     & 59    & 726   & 4334  & 42575     \\ 
         Objective value        & 1.428 & 1.445 & 2.028 & 2.041 & 2.438  \\\bottomrule
    \end{tabular}    
\end{table}
\begin{figure}
    \centering
    \includegraphics[width=\linewidth]{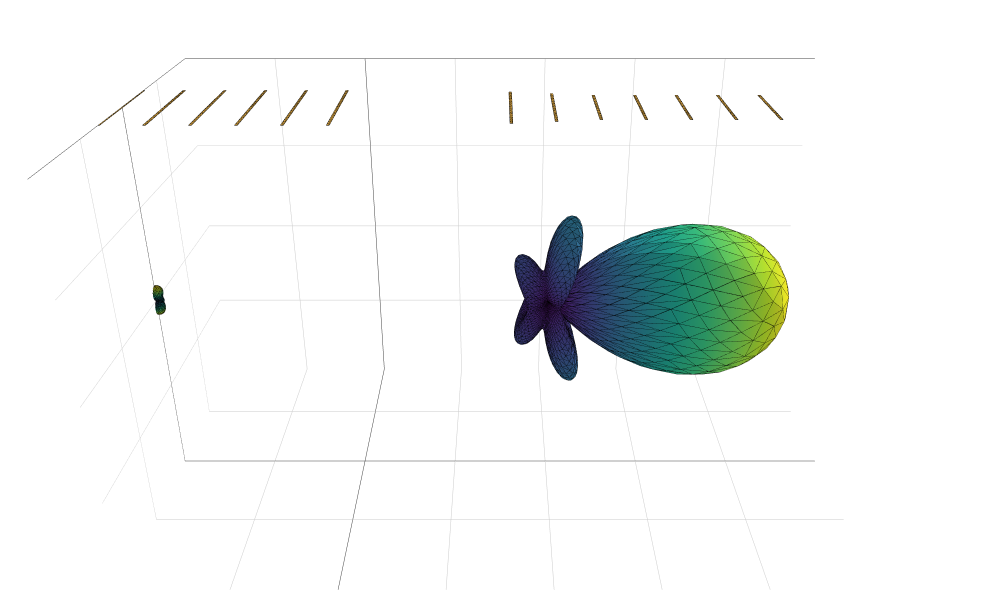}
    \caption{Initial and optimized design and radiation pattern of the Yagi--Uda antenna test case with seven elements. Note that the radiation from the initial design is very weak compared to the optimized design.}
    \label{fig: yagi results}
    
\end{figure}

\section{Discussion}
\label{section: discussion}

The main contribution of this paper is to give evidence that an optimization problem that by its nature is highly  non-linear and inverse can  be solved to optimality by a single mixed integer program, in reasonable computation time. The fact that this yields optimality guarantees contrasts our work with all previous publications that address antenna design problems. 

While  the paper focuses on the optimal design of antenna topologies, we believe that our approach can serve as a blueprint also for other nonlinear and inverse design problems, say in fluid mechanics or acoustics (the latter field even has direct analogues of many of the concepts we discuss here). That being said, the purpose of this paper is also to open a new source of discrete optimization challenges for the mathematical optimization community. 

This paper is the first in a series; the second paper~\cite{hoeksma2026physics} develops our method further to solve more complicated antenna design problems with more constraints, and develops techniques that allow larger problems to be solved in reasonable time.

Next to these two papers, which focus on the mathematical fundamentals of our approach, there is a third companion paper~\cite{vanrhijn2026optimal} that addresses the same topic but has a different focus. It targets the antenna engineering community,  and  provides three  elaborate use cases to highlight the promise that lies in using discrete optimization methodology, and mixed integer programming in particular. 

That companion paper has a focus on practically relevant design questions such as impedance matching and polarization. For the latter, the mathematical background is given in~\cite{hoeksma2026physics}. It also features a  comparison to an existing competitive Yagi--Uda design.

An obvious direction for future research is to generate more ideas for scalability of the mixed integer programming approach. 
That will be necessary for example when considering more complicated antenna designs, and especially when optimizing not just a single antenna but a whole array of antennas, since nearby radiating antennas interact with each other.

\begin{acknowledgements}
We are grateful to Dave Bekers, Raymond van Dijk, Stefania Monni and Yannis Iliopoulos at the Netherlands Organisation for Applied Scientific Research, who provided us with much help and feedback over the course of this project. 

The authors received funding from the Dutch National Growth Fund programme POLARIS (\url{www.polaris-ngf.nl}).
\end{acknowledgements}


\section*{Conflict of interest and AI}

The authors declare that they have no conflict of interest. This paper and the results therein were obtained without substantial usage of AI tools. AI was exclusively used for some routine calculations, to help with literature search, and to help with programming.


\printbibliography

\begin{appendix}
    \section{Proof of \texorpdfstring{\Cref{thm:charnescooper}}{Theorem 3.2}}
    \label{app:charnescooper}
    \charnescooper*
    \begin{proof}
    We will first prove the first statement.
    Suppose that $x^*$ is not optimal for \eqref{eq: problem R}, that is, there exists an $x' \in R$ with 
    \[
    \left|\frac{c^Tx' + \alpha}{d^Tx' + \beta}\right| > \left|\frac{c^Tx^* + \alpha}{d^Tx^* + \beta}\right|.
    \]
    The above can only hold if $d^Tx' + \beta \neq 0$ and $c^Tx' + \alpha \neq 0$.
    So let 
    \begin{align*}
        y' = \frac{x'}{|d^Tx' + \beta|}\frac{\overline{(c^Tx' + \alpha)}}{|c^Tx' + \alpha|},\\
        z' = \frac{1}{|d^Tx' + \beta|}\frac{\overline{(c^Tx' + \alpha)}}{|c^Tx' + \alpha|}
    \end{align*}
    Then $(y', z') \in S$, since
    \begin{align*}
        Ay' = \frac{Ax'}{|d^Tx' + \beta|}\frac{\overline{(c^Tx' + \alpha)}}{|c^Tx' + \alpha|} = \frac{b}{|d^Tx' + \beta|}\frac{\overline{(c^Tx' + \alpha)}}{|c^Tx' + \alpha|} = bz'\\
        |d^Ty' + \beta z'| = \left|\frac{(d^Tx' + \beta)}{|d^Tx' + \beta|}\frac{\overline{(c^Tx' + \alpha)}}{|c^Tx' + \alpha|}\right| = 1.
    \end{align*}
    But now
    \begin{align*}
        \Re\{ c^Ty' + \alpha z' \} 
        &= \Re \left\{ \frac{(c^Tx' + \alpha)}{|d^Tx' + \beta|}\frac{\overline{(c^Tx' + \alpha)}}{|c^Tx' + \alpha|} \right\} 
        = \Re\left\{\frac{|c^Tx'+\alpha|^2}{|d^Tx' + \beta||c^Tx' + \alpha|}\right\} 
        = \left|\frac{c^Tx' + \alpha}{d^Tx' + \beta}\right|\\
        &> \left|\frac{c^Tx^* + \alpha}{d^Tx^* + \beta}\right|
        = \left|\frac{c^Ty^* + \alpha z^*}{d^Ty^* + \beta z^*}\right|
        = |c^T y^* + \alpha z^*| \geq \Re\{c^Ty^* + \alpha z^*\},
    \end{align*}
    so $(y^*, z^*)$ is not optimal for \eqref{eq: problem S}, contradicting our assumption.
    Therefore, $x^*$ must be optimal for \eqref{eq: problem R}.

    Now, since $\arg z$ is free, in an optimal solution it will be so chosen such that the real part of the objective equals its absolute value, which implies
    \[
        \Re \{c^Ty^* + \alpha z^*\} = |c^Ty^* + \alpha z^*| = \frac{|c^Ty^* + \alpha z^*|}{|d^Ty^* + \beta z^*|} = \left|\frac{c^Tx^* + \alpha}{d^Tx^* + \beta}\right|.
    \]

    Now, we prove the second statement by cases. Let $(y^*, z^*)$ be an optimal solution to \eqref{eq: problem S} with $|d^Ty^* +\beta z^*| < 1$. Clearly, $\Re\left\{c^Ty^* + \alpha z^*\right\} \geq 0$, since otherwise $(-y^*, -z^*)$ is a feasible solution with better objective. 

    \begin{description}
        \item[Case 1:] $d^Ty^* +\beta z^* \neq 0$ and $z^* \neq 0$. Then
        \[
            (y', z') = \left(\frac{y^*}{|d^Ty^* +\beta z^*|}, \frac{z^*}{|d^Ty^* +\beta z^*|}\right)
        \]
        is also feasible with objective
        \[
            \Re\left\{c^Ty' + \alpha z'\right\} =  \Re\left\{\frac{c^Ty^* + \alpha z^*}{|d^Ty^* +\beta z^*|}\right\} \geq \Re\left\{c^Ty^* + \alpha z^*\right\},
        \]
        where the above inequality becomes strict whenever $\Re\left\{c^Ty^* + \alpha z^*\right\} > 0$, contradicting optimality of $(y^*, z^*)$. In the case where $\Re\left\{c^Ty^* + \alpha z^*\right\} = 0$,  $(y', z')$ is also optimal for \eqref{eq: problem S} with $|d^Ty' +\beta z'| = 1$.
        
        \item[Case 2:] $d^Ty^* +\beta z^* \neq 0$ and $z^* = 0$. As in the above case, we must have $\Re\left\{c^Ty^* + \alpha z^*\right\} = 0$. By optimality of $(y^*, z^*)$, there cannot exist an $x$ with 
        \[
            Ax = b, |c^Tx + \alpha| > 0, |d^Tx + \beta| > 0,
        \]
        since otherwise we could use such a solution (call it $x'$) to construct 
        \begin{align*}
            y' = \frac{x'}{|d^Tx' + \beta|}\frac{\overline{(c^Tx' + \alpha)}}{|c^Tx' + \alpha|},\\
            z' = \frac{1}{|d^Tx' + \beta|}\frac{\overline{(c^Tx' + \alpha)}}{|c^Tx' + \alpha|}.
        \end{align*}
        As before, this is feasible for \eqref{eq: problem S}, and has a positive objective value, so that
        \[
            \left|\frac{(c^Tx'+\alpha)}{|d^Tx' + \beta|}\frac{\overline{(c^Tx' + \alpha)}}{|c^Tx' + \alpha|}\right| =|c^Ty' + \alpha z'|\geq \Re\{c^Ty' + \alpha z'\} > \Re\{c^Ty^* + \alpha z^*\} = 0,
        \]
        contradicting optimality of $(y^*, z^*)$.
        So either the optimal objective value of \eqref{eq: problem R} is 0, or \eqref{eq: problem R} is infeasible.
        
        \item[Case 3] $d^Ty^* +\beta z^* = 0$.
        Then $Ay^* = bz^*$, so $A(\lambda y^*) = b(\lambda z*)$  and $d^T(\lambda y^*) + \beta\lambda z^* = 0$ for any $\lambda \in \C$. Then either
        \begin{description}
            \item[(a)] $|c^Ty^* + \alpha z^*| > 0$. Then \eqref{eq: problem S} is unbounded, contradicting optimality of $(y^*, z^*)$.
            \item[(b)] $c^Ty^* +\alpha z^*= 0$. Then, as before, either the optimal objective value of \eqref{eq: problem R} is 0, or \eqref{eq: problem R} is infeasible.
        \end{description}
    \end{description}
    
    Finally, we prove that unboundedness of \eqref{eq: problem S} implies unboundedness of \eqref{eq: problem R}.
    If \eqref{eq: problem S} is unbounded, then there exist $(y', z')$ such that 
    \begin{align*}
         \Re\{c^Ty' + \alpha z'\} > 0, && Ay' = bz', && d^Ty'+ \beta z' = 0.
    \end{align*}
    We will again proceed by cases.
    \begin{description}
        \item[Case 1] $z' \neq 0$. Then $x' = y'/z'$ is valid for \eqref{eq: problem R} with $Ax' = b$ and $d^Tx'+\beta = 0$.
        If there exists an $\hat{x}$ such that $A\hat{x} = b$, $d^T\hat{x}+\beta \neq 0$, then \eqref{eq: problem R} is unbounded, as all of the convex combinations
        \[
        \lambda x' + (1-\lambda)\hat{x}, \quad \lambda \in [0, 1)
        \] 
        are valid solutions of \eqref{eq: problem R}, and
        \[
        \lim_{\lambda\rightarrow1} \left|\frac{c^T( x' + (1-\lambda)\hat{x}) +\alpha}{d^T(\lambda x' + (1-\lambda)\hat{x}) +\beta}\right| =  \lim_{\lambda\rightarrow1} \left|\frac{\lambda (c^Tx' + \alpha) + (1-\lambda)(c^T\hat{x} +\alpha)}{(1-\lambda)(d^T\hat{x} +\beta)}\right| = \infty\,.
        \]
        So \eqref{eq: problem R} is infeasible or unbounded.
        \item[Case 2:] $z' = 0$. Now $y'$ is such that
        \begin{align*}
         \Re\{c^Ty'\} > 0 && Ay' = 0 && d^Ty' = 0,
        \end{align*}
        so $y'$ is an unbounded ray for \eqref{eq: problem R}, and therefore \eqref{eq: problem R} is either infeasible or unbounded.
    \end{description}
\end{proof}
\end{appendix}

\end{document}